\documentclass{bcp}
\usepackage{amsmath,amsthm}
\usepackage{amssymb}

\usepackage[utf8]{inputenc}
\usepackage[T1]{fontenc}
\usepackage[english]{babel}
\usepackage{graphicx}
\usepackage{url}

\newtheorem{theorem}{Theorem}[section]
\newtheorem{prop}[theorem]{Proposition}

\newtheorem{maintheorem}[theorem]{Main Theorem}

\theoremstyle{definition}
\newtheorem{defi}[theorem]{Definition}

\newtheorem{examples}[theorem]{Examples}
\newtheorem{remark}[theorem]{Remark}
\newtheorem{comments}[theorem]{Comments}

\newcommand{\field}[1]{\mathbb{#1}}

\newcommand{\RR}{\field{R}}

\def\vect#1{\overrightarrow{#1}}

\def\d{\mathrm{d}}

\begin{document}

\mathclass{Primary 53D05, secondary 53D12, 53B50, 7803.}

\abbrevauthors{C.-M. Marle}
\abbrevtitle{Hamilton and the Malus-Dupin theorem}

\title{The works of William Rowan Hamilton\\ 
       in Geometrical Optics\\
       and the Malus-Dupin theorem}

\author{Charles-Michel Marle}
\address{Universit\'e Pierre et Marie Curie\\
Paris, France\\
Personal address: 27 avenue du 11 novembre 1918, 92190 Meudon, France\\
E-mail: charles-michel.marle@math.cnrs.fr, cmm1934@orange.fr}

\maketitlebcp

\begin{abstract}
The works of William Rowan Hamilton in Geometrical Optics are presented,
with emphasis on the Malus-Dupin theorem. According to that theorem, a family of light rays depending on two
parameters can be focused to a single point by an optical instrument made of reflecting or refracting surfaces
if and only if, before entering the optical instrument, the family of rays is rectangular (\emph{i.e.}, admits 
orthogonal surfaces). Moreover, that theorem states that a rectangular system of rays remains rectangular
after an arbitrary number of reflections through, or refractions across, smooth surfaces of arbitrary shape.
The original proof of that theorem due to Hamilton is presented, along with another proof founded in symplectic geometry. It was the proof of that theorem which led Hamilton to introduce his 
\emph{characteristic function} in Optics, then 
in Dynamics under the name \emph{action integral}.
\end{abstract}

\section{Introduction}

It was a pleasure and a honour to present this work at the international meeting \lq\lq Geometry of Jets and Fields\rq\rq\ in honour of Professor Janusz Grabowski.
\par\smallskip

The works of Joseph Louis Lagrange (1736--1813) and Siméon Denis Poisson (1781--1840) during the years
1808--1810 on the slow variations of the orbital elements of planets in the solar system, are one of the main sources of contemporary Symplectic Geometry. Another source of Symplectic Geometry, equally important 
but maybe not so  well known, are the works on Optics due to
Pierre de Fermat (1601--1665),  Christian Huygens (1629--1695), \'Etienne-Louis Malus (1775--1812), 
Charles François Dupin (1784--1873) and William Rowan Hamilton (1805--1865). 
I will deal mainly in what follows with the Malus-Dupin theorem in Geometrical Optics, according to which
a family of light rays smoothly depending on two parameters which, when entering an optical system, has a property called \emph{rectangularity}\footnote{It is the terminology used by Hamilton.}, keeps that property
in each transparent medium in which it propagates. The optical system may be made of any number of homogeneous and isotropic transparent media of various refractive indices separated by smooth surfaces and of reflecting smooth surfaces of arbitrary shapes.
\par\smallskip

In view of explaining what is \emph{rectangularity} of a family of light rays, let me recall a few basic concepts of Geometrical Optics.
\par\smallskip

In Geometrical Optics, luminous phenomena are described in terms of \emph{light rays}.
In classical (non-relativistic) Physics, the physical Space in which we live and in which light propagates
is mathematically described as an affine three-dimensional space $\mathcal E$ endowed, once a unit of length  has been chosen, with an Euclidean structure. The set of all oriented straight lines in $\mathcal E$ will be denoted by $\mathcal L$. In an homogeneous and isotropic transparent medium, a \emph{light ray} is a connected component of the part of an oriented straight line contained in that medium, hence a segment of an element of
$\mathcal L$. For dealing with reflections and refractions, it will be convenient to consider the full oriented straight line which bears that segment. A reflection on a mirror mathematically described as a smooth surface, or a refraction acroos a smooth surface which separates two transparent media with different refractive indices, are therefore described as \emph{transformations} of the space $\mathcal L$,
\emph{i.e.} as maps defined on an open subset of $\mathcal L$, wich associates to each light ray which hits the reflecting or refracting surface the correponding reflected or refracted light ray.
\par\smallskip

It will be proven below that $\mathcal L$ can be very naturally endowed with a four-dimensional smooth manifold structure.

\begin{defi}\label{FamilleRayons}
A \emph{family of light rays smoothly depending on $n$ parameters} ($1\leq n\leq 4$) is an immersed
(not necessarily embedded) submanifold $\mathcal F$ of dimension $n$ of  $\mathcal L$. In short, it will be called \emph{an $n$-parameters family of rays}, the smoothness being tacitly assumed.
\end{defi}

\begin{defi}\label{PointRegulier}
A \emph{regular point} of a light ray $L_0$ in a two-parameters family $\mathcal F$ of rays is a point 
$m_0\in L_0$ with the following property: for any smooth surface $S\subset \mathcal E$ containing $m_0$  
and transverse at that point to the ray $R_0$, there exists an open neighbourhood $U$ of $L_0$ in 
$\mathcal F$ and an open neighbourhood $V$ of $m_0$ 
in $S$, such that each ray $L\in U$ meets $V$ at a unique point $m$ and that the map 
$L\mapsto m$ is a diffeomorphism of $U$ onto $V$.  
\end{defi}

\begin{defi}\label{NormalFamily} 
A two-parameters family $\mathcal F$ of rays is said to be \emph{rectangular} when for each regular point 
of each ray $L\in{\mathcal F}$, there exists  a smooth surface orthogonally crossed by $L$
and by all the rays in a neighbourhood of $L$ in $\mathcal F$.
\end{defi}

I can now indicate a mathematically precise statement of the Malus-Dupin theorem.

\begin{maintheorem}[Malus-Dupin theorem]\label{TheoremeMalusDupin}
A rectangular family of light rays which enters into an optical system with any number of
smooth reflecting of refracting surfaces remains rectangular in each homogeneous and isotropic
transparent medium in which it propagates.
\end{maintheorem}

\begin{comments}\label{Commentaires}
Regularity of a point on a ray in a two-parameters family, and rectangularity of a two-parameters family of rays, are local properties. When a point $m_0$ on a ray $R_0$ in a two-parameters family of rays $\mathcal F$
is regular, there exists an open neighbourhood $U$ of $R_0$ in $\mathcal F$ and an open neighbourhood $W$
of $m_0$ in $\mathcal E$ such that each ray in $U$ meets $W$ and that each point in $W$ belongs to a unique ray in $U$ and is regular on that ray. By taking, for each point $m\in W$, the plane through $m$ orthogonal
to the ray $R\in U$ which bears that point, one obtains a rank 2 distribution on $W$. The rectangularity 
of the family of rays $\mathcal F$ means that all the distributions so obtained are integrable
in the sense of \cite{Sternberg} definition 5.2 p. 130.
\par\smallskip

Very often, a two-parameter family of rays is such that in some parts of the physical space 
$\mathcal E$ several sheets of that family of rays are superposed. The set of non-regular points 
on rays of such a family constitute the \emph{caustic surfaces} of the family. These surfaces were studied by Hamilon as soon as 1824, when he was 19 years old \cite{Hamilton1824}\footnote{Hamilton assumes implicitely
that on each ray of a two-parameter family there exists regular points. I will not give a proof of 
this property, nor will I study the caustic surfaces. Readers interested in these advanced topics related to the theory of singularities are referred to
\cite{Arnold}, chapter 9, section 46, pages 248--258 and Appendix 11, pages 438--439. They will find
in that book and in \cite{GuilleminSternberg1984} (Introduction, pp. 1--150) 
applications of symplectic Geometry in Geometrical Optics much more advanced than those discussed here.}.
\end{comments}

\begin{examples}\label{ExemplesFamillesRectangulaires}
The family of rays emitted by a luminous point in an homogeneous  and isotropic transparent medium
is rectangular, since all the spheres centered on the luminous point are orthogonal to all rays. On each ray, any point other than the luminous point is regular.
\par\smallskip

Similarly, the family of rays emitted in an homogeneous and isotropic transparent medium by a 
smooth luminous surface, when each point of that surface emits only one ray in a direction orthogonal 
to the surface, is rectangular: it is indeed a well known geometric promerty of the family of 
straight lines orthogonal to a smooth surface. 
\par\smallskip

In the three-dimensional affine Euclidean space $\mathcal E$, let $D_1$ and $D_2$ be two straight lines, orthogonal to each other, which have no common point. Let $\mathcal F$ be the two-parameters family of straight lines which meet both $D_1$ and $D_2$, oriented from $D_1$ to $D_2$. Frobenius' theorem
(\cite{Sternberg} theorem 5.2 p. 134) proves that $\mathcal F$
is not rectangular.
\end{examples}

After a short presentation of the historical background of the Malus-Dupin theorem, I will explain
its original proof due to Hamilton. Then I will present another proof founded in Symplectic Geometry.
 
\section{Historical background}\label{history}

\'Etienne Louis Malus de Mitry (1775--1812) 
was a soldier in the French army, a mathematician and a physicist. He studied the properties of families of oriented straight lines in view of applications in Optics.
Moreover, he developed the undulatory theory of light due to Christian Huygens (1629--1695), 
discovered and studied the phenomena of \emph{light polarization} and of
\emph{birefringence} which occurs when light propagates in some crystals.
He was engaged in the disastrous military campaign launched by Napoléon in Egypt (1798--1801). In Egypt he fell ill of a terrible disease, the plague, and was miraculously cured.
In 1811 he became Director of Studies at the French \'Ecole Polytechnique. Weakened by the diseases caught in Egypt, he died of tuberculosis in 1812. During the Egypt campaign he kept a journal which was published   
eighty years after his death \cite{Malus1892}.
\par\smallskip

Malus proved \cite{Malus1808} that the family of rays emitted by a luminous point (which, as seen above,
is rectangular) still is rectangular after \emph{one reflection} on a smooth mirror or
\emph{one refraction} across a smooth surface. But he was in doubt whether this property is still
satisfied for \emph{several} successive reflections or refractions \cite{Malus1811}
\footnote{This is a very nice example of application of the famous Arnold's theorem: when a theorem or a mathematical object is named by a person's name, that person \emph{is not} the person who proved that theorem
or who created that mathematical object. V.~Arnold used to add: \emph{of course, my theorem applies to itself!}}.
Malus' works on families of oriented straight lines were later used and much extended by Hamilton
\cite{Hamilton1824, Hamilton1827, Hamilton1830, Hamilton1831, Hamilton1837}.
\par\smallskip

Charles François Dupin (1784--1873) was a French naval engineer and mathematician. His name is linked to
several mathematical objects: \emph{Dupin's cyclids}, remarkable surfaces he discovered when he
was a youg student of Gaspard Monge (1746--1818) at the French \'Ecole Polytechnique; 
\emph{Dupin's indicatrix} which describes the shape of a smooth surface near one of its point.
I think that Arnold's theorem stated in the footnote below does not apply to these objects, which are indeed due to Dupin. He spent several years in Corfu (Greece) where he renovated the naval dockyard, while
participating in the creation, then in the works of the Ionian Academy. He became Professor
at the French \emph{Conservatoire des Arts et Métiers}, where he lectured and wrote several books 
for the education of working classes. He had very modern ideas: he thought that young girls should receive as good an education as youg boys and believed that a general increase of the education level would have beneficial effects on the whole society. Unfortunately these generous ideas are still not everywhere 
in application in today's world. He was exceptionnally shrewd: according to Wikipedia \cite{Dupin},
he inspired the poet and novelist \emph{Edgar Allan Poe} 
(1809--1849) the character of \emph{Auguste Dupin} appearing in the three 
detective stories \emph{The murders in the rue Morgue}, 
\emph{The Mystery of Marie Roget} and \emph{The Purloined Letter}. He found a very
neat geometric proof of the Malus-Dupin theorem for reflections \cite{Dupin1816} and he knew that the same result was true for refractions but did not publish his proof. 
\par\smallskip

According to \cite{ConwaySynge}, Adolphe Quetelet (1796--1874) and Joseph Diaz Gergonne (1771--1859)
obtained in 1825 a proof of the Malus-Dupin theorem both for reflections and for refractions. A little later,
the great Irish mathematician William Rowan Hamilton
(1805--1865) independently obtained a complete proof of that theorem \cite{Hamilton1827}. He knew the
previous works of Malus on the subject and quoted them in his own works, but it seems that he did not knew
the works of Dupin, Quetelet and Gergonne. Maybe this explains why that theorem, called in French textbooks on Optics the \emph{Malus-Dupin theorem} \cite{Courty}, is 
generally called \emph{Malus' theorem} in other countries.

\section{Hamilton's proof of the Malus-Dupin theorem}
I present in this section Hamilton's proof of the Malus-Dupin theorem 
(\cite{Hamilton1827}), first for a reflection, then for a refraction. While scrupulously following
Hamilton's ideas, I use today's vector notations in use in mathematics and physics.
Moreover I use figures to illustrate Hamilton's reasoning, although there are none 
in his publications~\footnote{According to the editors of Hamilton's mathematical works, the
total lack of figures in his published works could be due to Lagrange's influence. Indeed 
Lagrange proudly writes in the preface of his famous book ~\cite{Lagrange5}: \lq\lq On ne trouvera point de
figure dans cet Ouvrage. Les méthodes que j'y expose ne demandent ni constructions, ni raisonnements géométriques ou méchaniques,
mais seulement des opérations algébriques, assujetties à une marche régulière et uniforme. Ceux qui aiment l'Analyse,
verront avec plaisir la Méchanique en devenir une nouvelle branche, et me sauront gré d'en avoir étendu ainsi le domaine\rq\rq.}.

\subsection{A Reflection}
Hamilton considers a two-parameters family $\mathcal F$ of rays reflected by a smooth surface $M$. He assumes that each ray in $\mathcal F$ meets $M$ transversally. His arguments are local since they apply to a small neighbourhood of each ray in $\mathcal F$. He proves successively three results. First, he proves that
if the reflected family of rays is focused to a single point, the incident family $\mathcal F$ is 
rectangular. Then he proves that if the incident family $\mathcal F$ is rectangular, for each ray in 
$\mathcal F$ one can choose the position and the shape of a smooth mirror in such a way that
by reflection that ray and all rays in a a small neighbourhood of it are focused to a single point. Moreover
the point at which the reflected rays  are focused can be almost freely
chosen, with very few restrictions. These two results are in a way  more precise that the Malus-Dupin theorem,
since they precisely indicate the nature of the reflected family. Finally, Hamilton proves the Malus-Dupin
theorem itself. 
\par\smallskip

\subsubsection{Mathematical formulae for reflections}\label{ExpressionLoisReflexion}
Let me first prove some formulae which follow from the laws of reflection. Let $k=(k_1,k_2)$
be local coordinates in the family of rays $\mathcal F$ defined on an open neighbourhood 
$V$ of $(0,0)$ in $\RR^2$.  
For each $k\in V$, the ray in $\mathcal F$ with coordinates $k$ will be denoted by $L_1(k)$.
Let $\vect{\mathstrut  u_1(k)}$ be its unitary directing vector, $P(k)$ be the point at which 
this ray hits the mirror $M$, 
$\vect{\mathstrut n(k)}$ be the unitary vector orthogonal to $M$ at $P(k)$ directed towards 
the reflecting side of that mirror,
$L_2(k)$ be the corresponding reflected ray and $\vect{\mathstrut u_2(k)}$ its unitary directing vector.
The maps which associate $P(k)$, $\vect{\mathstrut u_1(k)}$ and $\vect{\mathstrut u_2(k)}$
to each $k\in V$ of course are smooth. Let $M_1(k)$ be a point on $L_1$ and $M_2(k)$ be a point
on $L_2(k)$ (figure 1). Their choice is up to now relatively free, it is only assumed that the maps
$k\mapsto M_1(k)$ and $k\mapsto M_2(k)$ are smooth. 
Let $O$ be a fixed point in the physical space $\mathcal E$ taken as origin. In short, 
$\vect{\mathstrut M_1(k)}$, 
$\vect{\mathstrut M_2(k)}$ and $\vect{\mathstrut P(k)}$ stand for the vectors
$\vect{\mathstrut O\,M_1(k)}$, $\vect{\mathstrut O\,M_2(k)}$ and $\vect{\mathstrut O\,P(k)}$.
We have

\begin{figure}[htp]\label{figure1}
\begin{center}
  \includegraphics{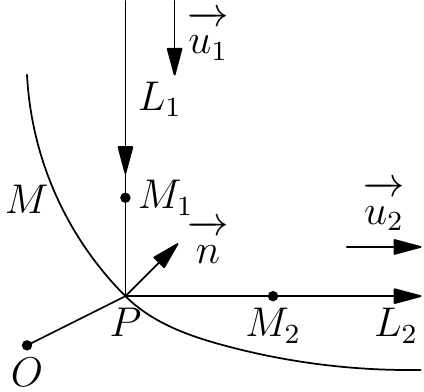}\\
  \caption{Reflection}
  \end{center}
\end{figure}

 $$\vect{\mathstrut M_1(k)\,P(k)}=\overline{\mathstrut  M_1(k)\,P(k)}\,\vect{\mathstrut u_1(k)}\,,\quad
     \vect{\mathstrut P(k)\,M_2(k)}=\overline{\mathstrut P(k)\,M_2(k)}\,\vect{\mathstrut u_2(k)}\,,
 $$
where $\overline{\mathstrut M_1(k)\,P(k)}$ and $\overline{\mathstrut P(k)\,M_2(k)}$ are the algebraic
values of the vectors 
$\vect{\mathstrut M_1(k)\,P(k)}$
and $\vect{\mathstrut P(k)\,M_2(k)}$, \emph{i.e} the lengths of the straight line segments 
$\bigl(M_1(k),P(k)\bigr)$ and $\bigl(P(k),M_2(k)\bigr)$ with the sign 
$+$ if the light propagates from $M_1(k)$ towards $P(k)$, or from $P(k)$ towards $M_2(k)$ 
(as on figure 1)
and $-$ in the reverse instance. The points $M_1(k)$ and $M_2(k)$ can indeed be chosen
behind the mirror $M$ on the straight lines which support the light rays $L_1(k)$ and $L_2(k)$.
By differentiating these equalities, we obtain
 \begin{align*}
 \d\vect{\mathstrut P(k)}-\d\vect{\mathstrut M_1(k)}&=\d\overline{\mathstrut  M_1(k)\,P(k)}\,\vect{\mathstrut u_1(k)}
                                                                                     +\overline{\mathstrut  M_1(k)\,P(k)}\,\d\vect{\mathstrut u_1(k)}\,,\\
\d\vect{\mathstrut M_2(k)}-\d\vect{\mathstrut P(k)}&=\d\overline{\mathstrut P(k)\,M_2(k)}\,\vect{\mathstrut u_2(k)}
                                                                                      +\overline{\mathstrut P(k)\,M_2(k)}\,\d\vect{\mathstrut u_2(k)}\,.
 \end{align*}
Let us take the scalar product with $\vect{\mathstrut u_1(k)}$ 
(respectively, with $\vect{\mathstrut u_2(k)}$) of both sides of the first (respectively second)
equality, and let us add the two equalities so obtained. 
Since the vectors $\vect{\mathstrut u_1(k)}$ and  $\vect{\mathstrut u_2(k)}$,
are unitary,  $\vect{\mathstrut u_1(k)}\cdot \d\vect{\mathstrut u_1(k)}=0$ and
$\vect{\mathstrut u_2(k)}\cdot \d\vect{\mathstrut u_2(k)}=0$ Therefore
 \begin{equation}
 \begin{split}
 \bigl(\vect{\mathstrut u_1(k)}-\vect{\mathstrut u_2(k)}\bigr)\cdot\d\vect{\mathstrut P(k)}&= 
 \vect{\mathstrut u_1(k)}\cdot\d\vect{\mathstrut M_1(k)}-\vect{\mathstrut u_2(k)}\cdot\d\vect{\mathstrut M_2(k)}\\
 &\quad+\d\bigl(\overline{\mathstrut  M_1(k)\,P(k)}+\overline{\mathstrut P(k)\,M_2(k)}\bigr)\,, 
 \end{split}
 \tag{*}
 \end{equation}
where the dot $\cdot$ stands for the scalar product of vectors.  
The laws of reflection show that the vectors $\vect{\mathstrut u_1(k)}-\vect{\mathstrut u_2(k)}$
and $\vect{\mathstrut n(k)}$ are parallel. Moreover any infinitesimal variation $\d\vect{\mathstrut P(k)}$ 
of the vector $\vect{\mathstrut P(k)}$
is tangent to the surface $M$ at $P(k)$, therefore is orthogonal to $\vect{\mathstrut n(k)}$. So
 $$\Bigl(\vect{\mathstrut u_1(k)}-\vect{\mathstrut u_2(k)}\Bigr)\cdot\d\vect{\mathstrut P(k)}=0\,.\eqno(**)$$
It follows from the above equalities $(*)$ $(**)$
 $$\vect{\mathstrut u_2(k)}\cdot\d\vect{\mathstrut M_2(k)}-
    \vect{\mathstrut u_1(k)}\cdot\d\vect{\mathstrut M_1(k)}
      =\d\Bigl(\overline{\mathstrut  M_1(k)\,P(k)}+\overline{\mathstrut P(k)\,M_2(k)}\Bigr)\,.\eqno({*}{*}{*})
 $$

\subsubsection{First result proven by Hamilton}\label{PremierResultatHamiltonReflexion}  
Hamilton assumes that the reflection on the mirror $M$ concentrates all the reflected rays 
onto a single point $M_2$, which can be real (before the mirror) or virtual (behind the mirror).
He chooses all the points $M_2(k)$ coincident with $M_2$. Therefore $\d\vect{\mathstrut M_2(k)}=0$, 
since $M_2(k)=M_2$ does not depend on $k$, so equality $({*}{*}{*})$ above becomes
 $$-\vect{\mathstrut u_1(k)}\cdot\d\vect{\mathstrut M_1(k)}
      =\d\Bigl(\overline{\mathstrut  M_1(k)\,P(k)}+\overline{\mathstrut P(k)\,M_2}\Bigr)\,.
 $$  
Hamilton chooses the point $M_1(k)$ on each incoming light ray $L_1(k)$ in such a way that
 $$\overline{\mathstrut  M_1(k)\,P(k)}+\overline{\mathstrut P(k)\,M_2}=\hbox{Constant}\,.
 $$
Then
 $$\vect{\mathstrut u_1(k)}\cdot\d\vect{\mathstrut M_1(k)}=0\,,
 $$
which proves that any infinitesimal variation of $M_1(k)$ is orthogonal to $\vect{\mathstrut u_1(k)}$. 
If the point $M_1(0,0)$ is regular on the ray $L_1(0,0)$, the points $M_1(k)$ draw  a small
smooth surface when $k$ varies around $(0,0)$, which is orthogonally crossed by the rays $L_1(k)$.
This proves that the family of rays $\mathcal F$ is rectangular in a neighbourhood of $L_1(0,0)$.

\subsubsection{Second result proven by Hamilton}\label{SecondResultatHamiltonReflexion}
Hamilton now assumes that the family of rays $\mathcal F$ is rectangular. He tacitly assumes that
there exists a regular point on the ray $L_1(0,0)$, and he takes that point for $M_1(0,0)$. 
The definition of rectangularity
\ref{NormalFamily} shows that there exists a small smooth open surface
$\Sigma$ containing $M_1(0,0)$ orthogonally crossed by the rays
$L_1(k)$ for all $k$ in some open neighbourhood $V'$ of $(0,0)$, $V'\subset V$. For each $k\in V'$
Hamilton chooses for $M_1(k)$ the point at which $L_1(k)$ crosses $\Sigma$.
\par\smallskip

Let $P(0,0)$ be any regular point on the ray $L_1(0,0)$. Of course it is possible to take
$P(0,0)=M_1(0,0)$ but any other regular point can be chosen. That point will be the point at which
$L_1(0,0)$ hits the mirror which will be constructed.
The comments \ref{Commentaires} prove that there exists an open neighbourhood $W$ of $P(0,0)$ in $\mathcal E$ 
whose all points are regular on the rays in $V'$ which cross them. By restricting eventually
$V'$ and $W$ it is possible to arrange things so that each  ray in $V'$ meets $W$ and that each 
point in $W$ is crossed by a unique ray in
$V'$. The map $X\mapsto k(X)$, which associates to each point $X\in W$ the coordinates
$k=(k_1,k_2)$ of the unique ray in $V'$ which crosses $X$, is smooth. Hamilton builds in $W$ 
the reflecting surface which concentrates the reflected rays to a single point.
\par\smallskip

Let $M_2$ be a point in the physical space $\mathcal E$ which will be the point on which
the reflected rays are focused. The only restriction on the choice of that point is that it must be 
other than $P(0,0)$. The two functions defined on $W$
 $$X\mapsto F_{\varepsilon}(X)=\overline{\mathstrut  M_1\bigl(k(X)\bigr)\,X}+\varepsilon\Vert \vect{\mathstrut X\,M_2}\Vert\,,
        \quad X\in W\,,\quad\hbox{with}\ \varepsilon=\pm 1 
 $$
are smooth and their first differentials at $P(0,0)$ vanish.
Consider the two subsets
 $$M_\varepsilon=\Bigl\{X\in W; F_\varepsilon(X)=F_\varepsilon\bigl(P(0,0)\bigr)\Bigr\}\,,\quad 
\varepsilon=\ \hbox{either} +\ \hbox{or}\ - \,.
 $$
By eventually restricting again $W$ and $V'$, things can be arranged so that the two subsets
$M_\varepsilon$ are two small open surfaces containing the point $P(0,0)$, transversally
met by all rays in $V'$.
If $M_2$ is situated on the ray $L_1(0,0)$ \emph{before} $P(0,0)$ (with respect to the orientation
of the oriented straight line $L_1(0,0)$), it necessarily will be a \emph{real} convergence point
of the reflected rays and one has to choose 
$\varepsilon=+$. If $M_2$ is on the ray $L_1(0,0)$ 
\emph{after} $P(0,0)$, it necessarily will be a \emph{virtual} convergence point of the reflected rays and one has to choose $\varepsilon=-$. 
For all other possible choices of $M_2$, the following calculations prove that both
$\varepsilon=+$ and $\varepsilon =-$ can be chosen. By taking the surface $M_\varepsilon$ as a
mirror, for any point
$P\in M_\varepsilon$,
 $$\varepsilon\Vert \vect{\mathstrut P\,M_2}\Vert=\overline{\mathstrut P\,M_2}$$
and for any infinitesimal variation of $P$ on $M_\varepsilon$,
 $$\d\Bigl(\overline{\mathstrut  M_1\bigl(k(P)\bigr)\,P}+\overline{\mathstrut P\,M_2}\Bigr)=0\,,\quad  
     \vect{\mathstrut u_1\bigl(k(P)\bigr)}\cdot\d\vect{\mathstrut M_1\bigl(k(P)\bigr)}=0\,.
 $$
Since $M_2\bigl(k(P)\bigr)=M_2$ does not depend on $P$, the equality $(*)$ above shows that 
 $$\Bigl(\vect{\mathstrut u_1\bigl(k(P)\bigr)}-\vect{\mathstrut u_2\bigl(k(P)\bigr)}\Bigr)\cdot\d\vect{\mathstrut P}=0\,.
 $$
This proves that for each $P\in M_\varepsilon$, reflection on the mirror $M_\varepsilon$ 
transforms the oriented straight $L_1\in{\mathcal F}$ which meets $M_\varepsilon$ at $P$
into the oriented straight line $L_2$ through the point $P$ with $\vect{\mathstrut u_2\bigl(k(P)\bigr)}$
as unitary directing vector.
\par\smallskip

\subsubsection{Proof of the Malus-Dupin theorem for a reflection}\label{MalusDupinHamiltonReflexion}
Hamilton finally proves the Malus-Dupin theorem fo a reflection: if a rectangular 
two-parameters family of rays hits transversally a smooth reflecting surface of any shape, the corresponding family of reflected rays is rectangular. For proving this result, he chooses a regular point $M_1(0,0)$ on the ray $L_1(0,0)$ and a smooth surface containing this point crossed orthogonally by the rays $L_1(k)$
for all $k$ near enough $(0,0)$. As point $M_1(k)$ on the incoming ray $L_1(k)$ he chooses the point at 
which that ray crosses orthogonally that surface, and on the corresponding reflected ray $L_2(k)$ he chooses
$M_2(k)$ so that 
 $$\overline{\mathstrut  M_1(k)\,P(k)}+\overline{\mathstrut P(k)\,M_2(k)}=\hbox{Constant}\,.
 $$
The above equality ${*}{*}{*}$ proves that
 $$\vect{\mathstrut u_2(k)}\cdot\d\vect{\mathstrut M_2(k)}=0\,.
 $$
If $M_2(0,0)$ is a regular point on the ray $L_2(0,0)$, the small displacements of $M_2(k)$ 
when $k$ varies  in a neighbourhood of $(0,0)$ draw a smooth surface orthogonally crossed 
by the rays $L_2(k)$. The family of reflected rays therefore is rectangular.

\subsection{A refraction}\label{HamiltonRefraction}
Hamilton considers a two-parameters family of rays $\mathcal F$ refracted across a smooth surface $R$. 
He assumes that at the point at which an incoming ray reaches the refracting surface $R$, that ray and the
corresponding refracted ray are transverse to that surface. He proves successively three results which correspond to those previously proven for a reflection, his proofs resting on the following
equalities.

\subsubsection{Mathematical formulae for a refraction}\label{ExpressionLoisRefraction}
The notations are the same as those in Section \ref{ExpressionLoisReflexion}:
$\mathcal F$ is a two-parameters family of rays refracted across a smooth refracting surface
$R$. Local coordinates on $\mathcal F$, denoted by $k=(k_1,k_2)$, take their values in a neighbourhood
$V$ of $(0,0)$ in $\RR^2$. For each $k\in V$, $L_1(k)\in{\mathcal F}$ is the ray of coordinates $k$,
$L_2(k)$ the corresponding refracted ray, $\vect{\mathstrut  u_1(k)}$ and
$\vect{\mathstrut  u_2(k)}$ their respective unitary directing vectors, $P(k)$ the point at which $L_1(k)$
meets the surface $R$, $\vect{\mathstrut n(k)}$ a unitary vector orthogonal to $R$ at $P(k)$, directed for example towards the side containing the ray $L_1(k)$. Let
$M_1(k)$ be a point on $L_1(k)$ and
$M_2(k)$ be a point on $L_2(k)$ chosen so that the maps $k\mapsto M_1(k)$
an $k\mapsto M_2(k)$ are smooth (Figure 2).

\begin{figure}[htp]\label{figure2}
\begin{center}
  \includegraphics{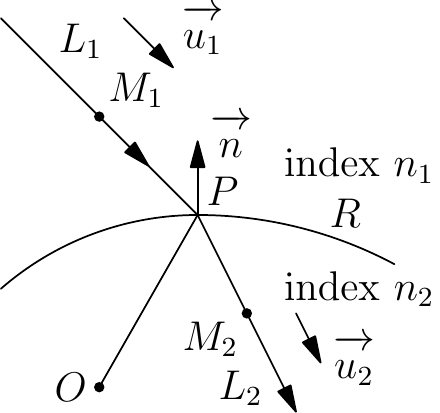}\\
  \caption{Refraction}
  \end{center}
\end{figure}  

Let $O\in{\mathcal E}$ be a fixed point taken as origin.
As before $\vect{\mathstrut M_1(k)}$, 
$\vect{\mathstrut M_2(k)}$ and $\vect{\mathstrut P(k)}$ stand for the vectors 
$\vect{\mathstrut O\,M_1(k)}$, 
$\vect{O\,M_2(k)}$ and $\vect{O\,P(k)}$. Let $n_1$ and $n_2$ be the refractive indices
of the two transparent media separated by the surface $R$. Arguments similar to those used
in section \ref{ExpressionLoisReflexion} for a reflection
easily lead to the following equality, which corresponds to equality $(*)$
of that section.
  \begin{equation*}
 \begin{split}
 \bigl(n_1\vect{\mathstrut u_1(k)}-n_2\vect{\mathstrut u_2(k)}\bigr)\cdot\d\vect{\mathstrut P(k)}&= 
 n_1\vect{\mathstrut u_1(k)}\cdot\d\vect{\mathstrut M_1(k)}-n_2\vect{\mathstrut u_2(k)}\cdot\d\vect{\mathstrut M_2(k)}\\
 &\quad+\d\bigl(n_1\overline{\mathstrut  M_1(k)\,P(k)}+n_2\overline{\mathstrut P(k)\,M_2(k)}\bigr)\,.
 \end{split}
 \end{equation*}

Equality $(**)$ obtained for a reflection must be replaced by
 $$\bigl(n_1\vect{\mathstrut u_1(k)}-n_2\vect{\mathstrut u_2(k)}\bigr)\cdot\d\vect{\mathstrut P(k)}=0\,,
 $$
which expresses Snell-Descartes' laws of refraction.
\par\smallskip
Using these two equalities, Hamilton briefly sketches the proof of the three following results,
which correspond to those he already obtained for a reflection. 

\subsubsection{First result proven by Hamilton} 
If a two-parameters family of rays is concentrated to a single point by refraction across a
smooth surface which separates two transparent media with different refractive indices,
before reaching that surface the family of ongoing rays is rectangular.

\subsubsection{Second result proven by Hamilton} 
Conversely, if a two-parameters family of rays contained in a transparent medium with refractive index
$n_1$ is rectangular, for any ray in this family one can choose the position and the shape
of a smooth refracting surface separating that transparent medium from another transparent medium
with refractive index $n_2\neq n_1$ so that after refraction all the
rays contained in some neighbourhood of that ray are concentrated to a single point.

\subsubsection{Proof of the Malus-Dupin theorem for refraction} 
When a rectangular family of rays contained in a transparent medium with refractive index $n_1$
is refracted, through a smooth surface, into another transparent medium of refractive index
$n_2\neq n_1$, the family of refracted rays is rectangular.

\subsection{The characteristic function} After proving the Malus-Dupin theorem 
for a reflection and for a refraction, Hamilton observes that this theorem remains valid
for an optical device with any number of smooth reflecting or refracting surfaces:
when a rectangular family of rays enters that optical device, the outgoing family of rays is 
also rectangular.
\par\smallskip

Hamilton introduces the notion of \emph{characteristic function} of an optical device. He successively
provides several, more and more general definitions of that notion 
\cite{Hamilton1827, Hamilton1830, Hamilton1831, Hamilton1837}. In its most general definition,
it is a function, $(M_1,M_2)\mapsto V(M_1,M_2)$ which depends of two points $M_1$
and $M_2$ taken in the part of physical space occupied by the device, equal to the 
\emph{optical length} of a light ray joining these two points, travelling into the
optical device and obeying the laws of reflection and of refraction at each encounter of a
reflecting or refracting smooth surface. For a device containing only reflecting surfaces, it is 
simply the sum of the lengths of all straight line segments constituting the light ray joining
$M_1$ and $M_2$. When the device contains refractive surfaces, it is the sum of the products
of the length of each straight line segment by the refracive index of the transparent medium 
which contains that segment. In \cite{Hamilton1837} Hamilton even considers optical devices 
made by a continuous transparent media whose refractive index may depend on the considered 
point in space, on the direction of the light ray and on a \emph{chromatic index} to allow 
the treatment of non-monochromatic light. The value $V(M_1,M_2)$ is then expressed by
an \emph{action integral} taken on the path of the light ray joining $M_1$ and $M_2$.
Applying this method to the propagation of light in birefringent crystals, Hamilton discovers 
the very remarkable phenomenon of \emph{conicar refraction}, whose existence was confirmed in 1833
by experiments, using an aragonite crystal, he suggested to his colleague Humphrey Lloyd 
of the  \emph{Trinity College} in Dublin.
\par\smallskip

Hamilton proves that the value of the action integral expressing $V(M_1,M_2)$ is
\emph{stationary} with respect to infinitesimal variations of the path on which that integral is calculated,
the end points $M_1$ and $M_2$ remaining fixed. This very important result relates Optics with the
calculus of variations, in agreement with the Principle stated in 1657 by Pierre de Fermat (1601--1665).
As well as in Optics, the characteristic function is used by Hamilton as a key concept in his famous
works on Dynamics \cite{Hamilton1834, Hamilton1835}.

\section{A symplectic proof of the Malus-Dupin theorem}\label{DemonstrationSymplectique}
It is proven in this section that the set $\mathcal L$ of oriented straight lines in an affine
Euclidean three-dimensional space $\mathcal E$ is endowed with a very naturally defined symplectic
form $\omega_{\mathcal L}$ (\ref{StructureVarieteDesRayons}) and that rectangular families of
oriented straight lines are immersed Lagrangian submanifolds of $({\mathcal L},\omega_{\mathcal L})$
(\ref{RectangulaireLagrangienne}). Then it is proven that any reflection on a smooth surface is 
a symplectomorphism of an open subset of $({\mathcal L},\omega_{\mathcal L})$ 
onto another open subset of that symplectic manifold (\ref{reflexionsymplectique}).
Similarly it is proven that any refraction through a smooth surface which separates
two transparent media with refractive indices $n_1$ and $n_2$ is a symplectomorphism
of an open subset of $({\mathcal L},n_1\omega_{\mathcal L})$ onto an open subset of
$({\mathcal L},n_2\omega_{\mathcal L})$. The Malus-Dupin theorem 
is an easy consequence of these results.

\begin{prop}\label{StructureVarieteDesRayons}
Let $\mathcal E$ be an affine Euclidean space of dimension $3$ and $\mathcal L$ 
be the set of oriented straight lines in $\mathcal E$.  The choice of a point $O\in{\mathcal E}$
determines a one to one map of $\mathcal L$ onto the cotangent bundle $T^*\Sigma$ to a sphere 
$\Sigma$ of dimension $2$. The pull-backs by this map of the topology, the differential 
manifold structure and the affine bundle structure of $T^*\Sigma$ endow $\mathcal L$ 
with a topology, a differential manifold structure and an affine bundle structure 
which do not depend on the choice of $O$. Moreover the pull-back of the exterior differential 
$\d\theta_\Sigma$ of the Liouville form $\theta_\Sigma$ on $T^*\Sigma$ is a symplectic form
$\omega_{\mathcal L}$ on $\mathcal L$ which does not depend on the choice of $O$.
\end{prop}

\begin{proof}
Let $\Sigma$ be a sphere of any radius $R$, for example $R=1$, centered on a point 
$C\in{\mathcal E}$, and let $O$ be another point in $\mathcal E$.
\begin{figure}
\begin{center}
  \includegraphics{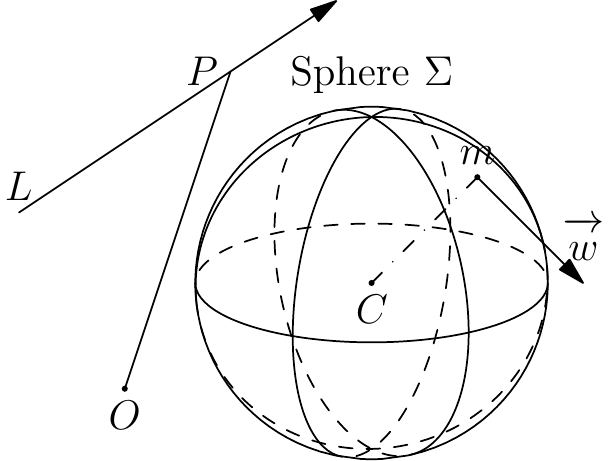}\\
  \caption{The space of oriented straight lines and the cotangent bundle to a sphere}
  \end{center}
\end{figure}
For each oriented straight line 
$L\in{\mathcal L}$, let $m_L\in\Sigma$ be the point such that
$\vect u_L=\vect{Cm_L}$ 
is a directing vector of $L$, and
$\eta_L$ be the linear form on the tangent space
$T_{m_L}\Sigma$ given by 
 $$\eta_L(\vect w)=\vect{OP}\cdot\vect w\,,\quad \vect w\in T_m\Sigma\,,
 $$
where $P$ is any point on $L$ and where the dot $\cdot$ stands for the scalar product of vectors.
(Figure 3).
Clearly the map $\Phi_O:L\mapsto\eta_L$ so defined is one to one from $\mathcal L$ onto $T^*\Sigma$.
That map does not depend on the choice of the centre $C$ of $\Sigma$ (if two spheres of the same radius
are identified by means of the translation which transports the centre of one sphere 
onto the centre of the other sphere). When $O$ is replaced by another point $O'\in{\mathcal E}$,
the covector $\eta_L\in T^*_{m_L}\Sigma$ is replaced by
$\eta'_L\in T^*_m\Sigma$, given by
$\eta'_L(\vect w)=\eta_L(\vect w)+\vect{O'O}\cdot{\vect w}$, where $\vect w$ is any vector in 
$T_{m_L}\Sigma$. The covector $\eta'$ can be expressed as
$\eta'_L=\eta_L+\d f_{O'O}(m_L)$,  where $f_{O'O}:\Sigma\to\RR$ is the smooth function
$f_{O'O}(n)=\vect{O'O}\cdot\vect{Cn}$, with $n\in\Sigma$.
The one to one maps $\Phi_{O'}$ and $\Phi_O$ therefore are related by
$\Phi_{O'}=\Psi_{O'O}\circ\Phi_O$, where
$\Psi_{O'O}:T^*\Sigma\to T^*\Sigma$ is the diffeomorphism 
$\eta\mapsto \eta+\d f_{O'O}\bigl(\pi_\Sigma(\eta)\bigr)$,
$\pi_\Sigma:T^*\Sigma\to\Sigma$ being the canonical projection.
\par\smallskip

The identification of $\mathcal L$ with $T^*\Sigma$ by means of the map $\Phi_0$ allows the transfer
on $\mathcal L$ of the mathematical structure of $T^*\Sigma$. So the set $\mathcal L$ 
becomes endowed with a topology, a differentiable manifold strucure, a vector bundle structure
and a $1$-form $\Phi_O^*\theta_\Sigma$, pull-back of the Liouville form $\theta_\Sigma$ on $T^*\Sigma$.
When $O$ is replaced by $O'$ the topology, the differentiable manifold structure and the affine bundle structure of $\mathcal L$ remain unchanged, while its vector bundle structure is modified; the 
$1$-form $\Phi_O^*\theta_\Sigma$ is replaced by 
 $$\Phi_{O'}^*\theta_\Sigma=\Phi_O^*(\Psi_{O'O}^*\theta_\Sigma)
                           =\Phi_O^*\bigl(\theta_\Sigma+\pi_\Sigma^*(\d f_{O'O})\bigr)\,.
 $$
Therefore $\Phi_{O'}^*(\d\theta_\Sigma)=\Phi_O^*(\d\theta_\Sigma)$ is a symplecic form 
$\omega_{\mathcal L}$ on $\mathcal L$ which does not depend on the choice of $O$.
\end{proof}

\begin{remark}\label{DroitesPointees}
Let $\widehat{\mathcal L}$ be the set of \emph{pointed oriented straight lines in} 
$\mathcal E$, \emph{i.e.} the set of pairs $(L,P)$ where $L\in{\mathcal L}$ is an oriented straight line
and $P$ is a point of $L$. This set is a smooth manifold of dimension $5$ which projects onto the manifold
$\mathcal L$ of dimension $4$, the projection $(L,P)\mapsto L$ amounting to \lq\lq forget\rq\rq\ 
the point $P$. Let $O$ be any fixed point in $\mathcal E$. An element $(L,P)$ in 
$\widehat{\mathcal L}$ can be represented by the pair of vectors
$(\vect{OP},\vect u)$ made by $\vect{OP}$ and the unitary directing vector 
$\vect u$ of the oriented line $L$. There 
exists\footnote{The choice of a point $O\in{\mathcal E}$ 
allows the identification of $\widehat{\mathcal L}$ with the restriction $T^*_\Sigma{\mathcal E}$
to the sphere $\Sigma$ of the cotangent bundle $T^*{\mathcal E}$. 
With this identification 
$\omega_{\widehat{\mathcal L}}$ is the form induced on $T^*_\Sigma{\mathcal E}$ by the canonical
symplectic form $\d\lambda_{\mathcal E}$ of the cotangent bundle $T^*{\mathcal E}$.}  
on $\widehat L$ an exact  differential $2$-form $\omega_{\widehat{\mathcal L}}$ given by
 $$\omega_{\widehat{\mathcal L}}(L,P)=\sum_{i=1}^3\d p_i\wedge\d u_i=\d\,\vect{OP}\wedge\d{\vect u}\,,
 $$
where $(p_1,p_2,p_3)$ and $(u_1,u_2,u_3)$ are the components of $\vect{OP}$ and
$\vect u$ in an orthonormal basis. The symbol $\wedge$ in the right hand side
is a combination of scalar and exterior products.
The form $\omega_{\widehat{\mathcal L}}$ projects onto $\mathcal L$ and its projection
is the symplectic form $\omega_{\mathcal L}$, which therefore can be expressed as
 $$\omega_{\mathcal L}(L)=\d\,{\vect{OP}}\wedge\d{\vect u}\,,\quad L\in{\mathcal L}\,,
 $$
since the right hand side does not depend on the choice of the point $P$ 
on the oriented straight line $L$, nor on that of the point $O$ in $\mathcal E$. 
This very convenient expression of
$\omega_{\mathcal L}$ is particularly well suited when used in conjunction with the usual
vector calculus in a three-dimensional Euclidean vector space. 
\end{remark}

\begin{prop}\label{RectangulaireLagrangienne} 
A two-parameters family of rays is rectangular in the sense of 
\ref{NormalFamily} if and only if it is an immersed Lagrangian submanifold
of the symplectic manifold $({\mathcal L},\omega_{\mathcal L})$ of oriented
straight lines in the affine Euclidean space $\mathcal E$.
\end{prop}

\begin{proof}
Each element $L_0$ in a two-parameters family  $\mathcal F$ of oriented straight lines
has an open neighbourhood in $\mathcal F$ which is the image of an injective smooth map
$L:(k_1,k_2)\mapsto L(k_1,k_2)$, defined on an open subset of $\RR^2$ containing $(0,0)$, 
with values in $\mathcal L$, such that $L(0,0)=L_0$. For each $k=(k_1,k_2)$ in the open subset of
$\RR^2$ on which the map $L$ is defined, let
$\vect u(k)$ be the unitary directing vector of $L(k)$ and $P(k)$ a point of $L(k)$.
The points $P(k)$ are not uniquely determined but it is always possible to choose them to make
smooth the map $k\mapsto\bigl(P(k),\vect u(k)\bigr)$. 
The remark \ref{DroitesPointees} allows us to write
 \begin{align*}
   L^*\omega_{\mathcal L}&=\d\bigl(\vect P(k)\cdot\,\d\vect u(k)\bigr)
                        =\d\Bigl(\d\bigl(\vect P(k)\cdot\vect u(k)\bigr)-\vect u(k)\cdot\,\d\,
    \vect P(k)\Bigr)\\
                     &=-\d\bigl(\vect u(k)\cdot\,\d\,\vect P(k)\bigr)\,,
 \end{align*}
where $\vect P(k)$ stands for the vector $\vect{OP}(k)$, $O$ being any fixed
point in $\mathcal E$.  
The immersed submanifold $\mathcal F$ is Lagrangian in a neighbourhood of
$L_0$ if and only if $L^*\omega_{\mathcal L}=0$
(\cite{LibermannMarle} p. 92, or \cite{OrtegaRatiu} p. 123), in other words 
if and only if the differential one-form  
$\vect u(k)\cdot \d\,\vect P(k)$ is closed.
Poincaré's lemma (\cite{Sternberg}, theorem 4.1 page 121) asserts that a one-form is closed if and only if
it is locally the differential of a smooth function. The immersed submanifold 
$\mathcal F$ therefore is Lagrangian near $L_0$ if and only if there exists a smooth
function $k\mapsto F(k)$, defined on a neighbourhood of $(0,0)$, such that 
 $$\vect u(k)\cdot\d\, \vect P(k)=\d F(k)\,.\eqno(*)
 $$
The vector $\vect u(k)$ being unitary, for any constant $c\in\RR$ we have, 
 $$\d F(k)=\vect u(k)\cdot\d\Bigl(\bigl(F(k)+c\bigr)\vect u(k)\Bigr)\,.
 $$
If a smooth function
$F$ satisfying $(*)$ exists, it also satisfies, for any constant $c\in\RR$,
 $$\vect u(k)\cdot\d\Bigl(\vect P(k)-\bigl(F(k)+c\bigr)\vect u(k)\Bigr)=0\,.\eqno(**)
 $$
Let us assume that $\mathcal F$ is Lagrangian in a neighbourhood of $L_0$, and let $F$
be a smooth function, defined in a neighbourhood of $(0,0)$, which satisfies $(*)$. 
Let $Q_0$ be a regular point on $L_0$. There exists a constant $c\in\RR$ such that
$\vect P(0,0)-\bigl(F(0,0)+c\bigr)\vect u(0,0)=\vect Q_0$, where $\vect Q_0$ stands for
the vector $\vect{\mathstrut OQ_0}$.
Since the points near enough $Q_0$ are regular on the rays near enough $L_0$ which cross them,
the variations of $\vect P(k)-\bigl(F(k)+c\bigr)\vect u(k)$ when $k$ varies around $(0,0)$ 
generate a smooth surface containing
$Q_0$ which, as shown by the equality $(**)$, is orthogonally crossed by the oriented straight lines
$L(k)$ for all $k$ near enough $(0,0)$. The family $\mathcal F$ therefore is rectangular
near $L_0$.
\par\smallskip

Conversely, let us assume that $\mathcal F$ is rectangular near $L_0$. As was tacitly done by Hamilton,
I assume that there exists a regular point on $L_0$. There
exists a smooth surface containing this point crossed orthogonally by $L_0$ and by the oriented 
straight lines $L(k)$ for all $k$ near enough $(0,0)$. This surface is made by the points
$P(k)-F(k)\vect u(k)$, when $k$ varies aroud $(0,0)$, $F$ being a smooth function.
The function $F$ satisfies equality $(*)$, therefore $\mathcal F$ is Lagrangian near $L_0$.
\end{proof}

\begin{prop}\label{reflexionsymplectique} 
Let $M\subset{\mathcal E}$ be a smooth open reflectig surface in the Euclidean three-dimensional
affine space $\mathcal E$. The set $U$ of oriented straight lines which meet $M$ 
transversally on its reflecting side is an open subset of the symplectic manifold 
$({\mathcal L},\omega_{\mathcal L})$ of all
oriented straight lines. The map which associates to each element in $U$ which bears a light ray
the oriented straight line which bears the corresponding reflected ray 
is a symplectomorphism of $U$ onto the open subset of $\mathcal L$
made by the straight lines in $U$ with the opposite orientation.
\end{prop}

\begin{proof}
Being determined by strict inequalities, $U$ is an open subset of $\mathcal L$. Let 
$L_1$ be a variable element in $U$, $P$ the point at which $L_1$ meets the reflecting surface
$M$, $L_2$ the oriented straight line which bears the corresponding reflected ray, 
${\vect u}_1$ and ${\vect u}_2$ the unitary directing vectors of $L_1$ and $L_2$, respectively
(figure 1).
As in Section \ref{ExpressionLoisReflexion}, $\vect P$ stands for the vector
$\vect{OP}$, $O\in{\mathcal E}$ being a fixed point taken as origin.
The expression of the symplectic form $\omega_{\mathcal L}$ given in Remark \ref{DroitesPointees}
shows that it is enough to prove the equality
$\d \vect P\wedge\d {\vect u}_2=\d\vect P\wedge \d{\vect u}_1$. 
The laws of reflection shows that ${\vect u}_2-{\vect u}_1=2({\vect u}_1\cdot{\vect n}){\vect n}$. Therefore
 \begin{align*}
 \d\vect P\wedge\d({\vect u}_2-{\vect u}_1)
 &=2\d\vect P\wedge\d\bigl((\vect u_1\cdot\vect n)\vect n\bigr)\\
 &=-2\d\bigl((\vect u_1\cdot\vect n)(\vect n\cdot\d\vect P)\bigr)\\
 &=0\,,
 \end{align*}
since $\vect n\cdot \d\vect P=0$, the vectors $\d\vect P$ and $\vect n$ being, respectively,
tangent to and orthogonal to the surface $M$ at $P$. 
\end{proof}

\begin{prop}\label{refractionsymplectique}
Let $R\subset{\mathcal E}$ be a smooth open refracting surface which separates two transparent media
with refractive indices $n_1$ and $n_2\neq n_1$, respectively.
The set $U$ of oriented straight lines which meet transversally the surface $R$ on the side of the medium
with refractive index $n_1$ under an angle such that the laws of refraction allow the existence, in 
the medium with refractive index $n_2$, of a refraced ray transverse to $R$, is an open subset $U$ of
the symplectic manifold $({\mathcal L},\omega_{\mathcal L})$ of all oriented straight lines. 
The map which associates
to each element in $U$ which bears a light ray
the oriented straight line which bears the corresponding refracted ray 
is a symplectomorphism of $U$ endowed with the symplectic form $n_1\omega_{\mathcal L}$  onto another
open subset of $\mathcal L$ endowed with the symplectic form $n_2\omega_{\mathcal L}$.
\end{prop}

\begin{proof}
The notations being the same as those in the proof of \ref{reflexionsymplectique}, let 
${\vect v}_1$ be the orthogonal projection of ${\vect u}_1$ onto the tangent plane to the surface 
$R$ at $P$. The Snell-Descartes' laws ofa refraction indicate that if a refracted ray $L_2$ corresponding to
$L_1$ exists, the orthogonal projection ${\vect v}_2$ of its unitary directing vector
${\vect u}_2$ onto the tangent plane to the surface 
$R$ at $P$ satisfies the equality
 $$n_2{\vect v}_2=n_1{\vect v}_1\,,\quad\hbox{or}
    \quad n_2\bigl({\vect u}_2-(\vect u_2\cdot\vect n)\vect n\bigr)
   =n_1\bigl({\vect u}_1-(\vect u_1\cdot\vect n)\vect n\bigr)\,.
 $$
When it can be satisfied that equality determines ${\vect u}_2$, therefore determines the
oriented straight $L_2$ which bears the refracted ray.
When  $n_1\leq n_2$ that equality can always be satisfied, but when $n_1>n_2$  it can be satisfied by a
straight line $L_2$ transverse to $R$ if and only if 
$({n_1}/{n_2})\Vert{\vect v}_1\Vert<1$, \emph{i.e.} if and only if the angle $\alpha_1$ made by $L_1$ with
the vector orthogonal to $R$ at $P$ satisfies $\displaystyle\sin\alpha_1<{n_2}/{n_1}$\footnote{
If that inequality is not satisfied the light ray supported by $L_1$ is totally reflected.}.
That inequality being strict, $U$ is an open subset of $\mathcal L$.
\par\smallskip

As in \ref{reflexionsymplectique} it is enough to prove that
$n_2\d \vect P\wedge\d {\vect u}_2=n_1\d\vect P\wedge \d{\vect u}_1$. 
We have
 \begin{align*}
 \d\vect P\wedge(n_2\d{\vect u}_2-n_1\d{\vect u}_1)
 &=\d\vect P\wedge\d\bigl(n_2(\vect u_2\cdot\vect n)\vect n
                         -n_1(\vect u_1\cdot\vect n)\vect n\bigr)\\
 &=-\d\Bigl(\bigl(n_2(\vect u_2\cdot\vect n)
                 -n_1(\vect u_1\cdot\vect n)\bigr)(\vect n\cdot\d\vect P)\Bigr)\\
 &=0\,,
 \end{align*}
since $\vect n\cdot\d\vect P=0$, the vectors $\d\vect P$ and $\vect n$ being, respectively,
tangent to and orthogonal to the surface
$R$ at $P$.
\end{proof}

\section{Conclusion}\label{Conclusion}

Reflections on and refractions across a smooth surface being symplectomorphisms, the propagation of light through an optical device made by several reflecting and refracting smooth surfaces
is a symplectomorphism (since by composition of several symplectomorphims one gets a symplectomorphism).
The image by a symplectomorphism of an immersed Lagrangian submanifold is another immersed
Lagrangian submanifold, which proves the Malus-Dupin theorem
\ref{TheoremeMalusDupin}.

\section{Acknowledgements}

I warmly thanks the organizers of the international conference \emph{Geometry of Jets and
Fields} for their kind invitation and their generous support during the conference, and I address
all my best wishes to Professor Janusz Grabowski for his birthday.

\end{document}